\documentclass[11pt,tbtags]{amsart}
\usepackage{amssymb}
\usepackage{amsmath}
\usepackage{amsthm}
\usepackage[swedish, english]{babel}
\usepackage[latin1]{inputenc}
\usepackage{psfrag}   
\usepackage{graphicx}
\usepackage{color}
\usepackage[square, comma, numbers, sort&compress]{natbib}
\usepackage{todonotes}

\theoremstyle{plain}
\newcommand{\E}{\mathbb E}
\newcommand{\R}{\mathbb R}
\newcommand{\C}{\mathcal C}
\newcommand{\D}{\mathcal D}
\newcommand{\cF}{\mathcal F}
\newcommand{\LL}{\mathcal L}
\newcommand{\cJ}{\mathcal J}
\newcommand{\ep}{\epsilon}

\def\P{{\mathbb P}}

\newcommand{\A}{\mathcal A}

\newenvironment{remark}[1][Remark]{\begin{trivlist}
\item[\hskip \labelsep {\bf Remark}]}{\end{trivlist}}
\newtheorem{theorem}{Theorem}[section]
\newtheorem{lemma}[theorem]{Lemma}
\newtheorem{corollary}[theorem]{Corollary}
\newtheorem{proposition}[theorem]{Proposition}

\theoremstyle{definition}

\title{The dividend problem with a finite horizon}
\author[Tiziano De Angelis and Erik Ekstr\"om]{Tiziano De Angelis and Erik Ekstr\"om$^1$}
\address{T.~De Angelis: School of Mathematics, University of Leeds, LS2 9JT Leeds, UK. }
\address{E.~Ekstr\"om: Department of Mathenatics, Uppsala University, Box 480, 75106 Uppsala, Sweden. }
\date{\today}
\email{t.deangelis@leeds.ac.uk, ekstrom@math.uu.se}
\thanks{$^1$ Support from the Swedish Research Council (VR) is gratefully acknowledged.}

\begin{document}

\begin{abstract} 
We characterise the value function of the optimal dividend problem with a finite time horizon as the unique classical solution of a suitable Hamilton-Jacobi-Bellman equation. The optimal dividend strategy is realised by a Skorokhod reflection of the fund's value at a time-dependent optimal boundary. 
Our results are obtained by establishing for the first time a new connection between singular control problems with an absorbing boundary and optimal stopping problems on a diffusion reflected at $0$ and created at a rate proportional to its local time.
\end{abstract}

\maketitle

\section{Introduction}

The dividend problem is a foundational problem in actuarial mathematics whose formulation dates back to De Finetti's work \cite{DeF}.
The model addresses the question of how a fund or an insurance company should distribute dividends to its beneficiaries prior to the time of ruin. After De Finetti's seminal work, the dividend problem has attracted the interest of many mathematicians and economists who produced a substantial body of literature on the subject. An extensive review of existing models and related mathematical results was published by 
Avanzi \cite{A09} in 2009, and the list of papers relative to the topic has continued to increase since.

Here we consider a canonical formulation of the problem in a simple diffusive setting that was proposed by Radner and Shepp \cite{RS96} and later considered also by \cite{JS} among many others. The value of a fund after dividends have been paid out evolves according to 
\begin{align*}
X^D_t=x+ \mu t + \sigma B_t -D_t,\qquad t\ge0,
\end{align*}
where $\mu$ and $\sigma>0$ are constants, $B$ is a Brownian motion and $D_t$ is the cumulative amount of dividends paid out
up to time $t$. 
The objective of the fund manager is to maximise the expected present value of future dividends up to the fund's default time $\gamma^D:=\inf\{t\geq 0:X^D_t\leq 0\}$. In addition, we also assume that the manager has a finite time horizon $T$ for the investment plan. 
The assumption of a finite horizon is the main difference between our model and the vast majority of the existing 
literature (including \cite{RS96} and \cite{JS}).
From the financial point of view, this restriction on the set of admissible dividend strategies is very natural and it simply means that an investment fund is liquidated at a pre-specified future date. 

If the fund's value at time $t\in[0,T]$ is $x>0$, the optimisation problem that the fund manager is faced with may be stated as follows:
\begin{equation}\label{P0}
\text{Find $D^*$ that maximises} \:\:\:\cJ(t,x;D):=\E\left[\int_{0-}^{\gamma^D\wedge (T-t)}e^{-rs}dD_s\right]
\end{equation}
where we integrate from $0\,-$ to account for the possibility of a jump of $D$ at time zero. 
From the mathematical point of view this is a problem of singular stochastic control (SSC) on a finite time horizon in which the underlying process is absorbed at zero. It is important to notice that zero is a regular point for the uncontrolled process $R_t=x+\mu t+\sigma B_t$ so that default may occur prior to $T$ with positive probability even if no dividends are distributed.

In this work we solve \eqref{P0} by constructing an optimal dividend strategy and by proving that the corresponding value function $V(t,x):=\cJ(t,x;D^*)$ is a classical solution of the Hamilton-Jacobi-Bellman equation, i.e.~in particular $V\in \mathcal C^{1,2}([0,T)\times(0,\infty))$. The optimal dividend strategy is shown to be the solution of a Skorokhod reflection problem at an appropriate time-dependent optimal boundary. 

To accomplish our task we develop a self-contained, fully probabilistic proof that hinges on a new type of connection between SSC problems and optimal stopping. Indeed, we show that $V_x=U$, where $U$ is the value function of an optimal stopping problem whose underlying process is a Brownian motion with drift $\mu$ and variance $\sigma^2$, which gets \emph{reflected at zero and created at a rate proportional to its local time} (cf.~\cite{P14}). Although links between optimal stopping and singular control have been known for many years (see for example \cite{BC67, BSW80, BenthReikvam, BoetiusKohlmann, BudhirajaRoss, DF, DF2, DeAFeMo15b, DeAFeMo15c, ElKK88, GT08, Karatzas85, KS, KS3, Taksar85} among others) our result makes a fundamental forward leap in this field. For the first time we establish that an absorbing boundary in SSC translates to a reflecting boundary with creation in optimal stopping. This new characterisation proves to be a powerful tool to tackle problem \eqref{P0} in an effective way.

We remark that, despite the vast existing literature on SSC, the study of problems that combine absorbing boundary behaviour with a finite time horizon is still a major theoretical challenge. For example, we observe that in the literature on optimal dividend problems an analytical characterisation of the optimal strategy and of the value function can only be found in models with infinite time horizon (for a theoretical study of problems of this kind one may refer to \cite{SLG}). These models are substantially easier to deal with compared to \eqref{P0} because they give rise to variational problems in the form of ordinary differential equations whereas our problem is associated to a parabolic one.

To the best of our knowledge, an analytical study of the problem in \eqref{P0} has only been addressed very recently by Grandits in a series of two papers, \cite{G1} and \cite{G2}, followed by a third one \cite{G} containing an extension of the canonical model. 
In these papers, Grandits uses methods from PDE and free-boundary analysis that rely on several transformations of the variational problem associated to \eqref{P0}. 
In \cite{G1}, the author obtains $\ep$-optimal boundaries, whereas in \cite{G2} he manages to pass to the limit as $\ep\to0$ and 
shows that the optimal strategy is of barrier type. Moreover, 
Grandits proves that the optimal boundary is continuous and that the value function $V$ is continuous with locally bounded weak derivatives $V_t$ and $V_{xx}$. Under strong assumptions on the regularity of the boundary he also derives asymptotic estimates for $t\to T$. 

Our new connection between the SSC problem \eqref{P0} and optimal stopping enables us to use powerful methods from optimal stopping theory in the study of the dividend problem. This leads to a self-contained probabilistic analysis that complements and improves results in \cite{G1} and \cite{G2}. We obtain spatial concavity and $\C^{1,2}$-regularity of the value function, monotonicity of the boundary and, without further assumptions, the boundary's asymptotic behaviour at $T$ along with its characterisation as the unique continuous solution of an integral equation (both 
these properties are actually consequences of results relative to the Russian option, see the last remark in Section \ref{sec:osp-intro} and results in Section \ref{sec:remarks}). 

The paper is organised as follows. In Section~\ref{sec:set-up} we introduce the dividend problem with a finite time horizon in some further detail, and 
we provide a verification theorem.
In Section~\ref{sec:osp-intro} we introduce a related 
optimal stopping problem with a peculiar boundary condition at $0$, and we state our main result, Theorem~\ref{main}, which shows the connection between 
these two problems. To prove Theorem~\ref{main}, we begin our study of the optimal stopping problem 
in Section~\ref{sec:osp} by proving continuity of the value function as well as existence and continuity of the optimal stopping boundary.
To apply a verification result, however, additional regularity of the optimal stopping problem is needed,
which is the main contribution of Sections~\ref{sec:Ux} (spatial regularity) and \ref{sec:Ut} (regularity in time). The proof of Theorem~\ref{main} is instead contained in Section \ref{sec:proof}.
Finally, Section~\ref{sec:remarks} gives a couple of concluding remarks concerning additional properties of the optimal boundary.

\section{The optimal dividend problem}
\label{sec:set-up}

Denote by $X^D=(X^D_s)_{s\in[0,\infty)}$ the value of a fund after dividends have been paid out according to 
a strategy $D$. We assume that
\begin{align}\label{XD}
X^D_s=x+ \mu s + \sigma B_s -D_s,
\end{align}
where $x\geq 0$, $\mu$ and $\sigma>0$ are constants, $B$ is a standard Brownian motion and $D$ is a non-negative, non-decreasing and 
right-continuous process (adapted to the filtration generated by $B$)
with the interpretation that $D_s$ represents the accumulated dividends paid out until time $s$. 
In particular, if $D_0>0$, then a lump sum $D_0$ is paid out at time 0. We only consider dividend strategies that 
satisfy $D_s-D_{s-}\leq X^D_{s-}$ at all times $s\in[0,\infty)$ (with a convention that $D_{0-}=0$), and we denote the set of 
such dividend strategies $\mathcal A$.

For a given dividend strategy $D\in\mathcal A$, denote by
\[\gamma^D:=\inf\{s\geq 0:X^D_s\leq 0\}\]
the (possibly infinite) default time of the firm, and consider the stochastic control problem
\begin{equation}
\label{V}
V(t,x)=\sup_{D\in\A} \E_x\left[\int_{0-}^{\gamma^D\wedge (T-t)}e^{-rs}dD_s\right].
\end{equation}
Here $\E_x[\,\cdot\,]=\E[\,\cdot\,|X^D_{0-}=x]$, $T>0$ is a given time horizon, and we refer to problem \eqref{V} as the {\bf dividend problem with finite horizon}.

\begin{remark}
The integral in \eqref{V} is interpreted in the Riemann-Stiltjes sense. In particular, the lower limit $0-$ of integration 
accounts for the contribution from an initial dividend payment $D_0>0$.
We also point out that choosing a strategy with $D_0=x$ in \eqref{V} yields the trivial inequality $V\geq x$.
\end{remark}

\begin{remark}
Notice that in \eqref{V} we consider the optimisation problem as if it were started at time $0-$ (i.e.~before dividends are paid) but with a time horizon equal to $T-t$. This is justified because $X^D$ is time-homogeneous and, if at time $t$ the fund's value before dividends are paid is $x>0$, then the residual time of the optimisation is $T-t$. 
Moreover we also point out that in the rest of the paper we use the time-space process $(t+s,X^D_s)_{s\in[0,T-t]}$ under the measure $\P_x$. The latter is equivalent to the process $(s,X^D_s)_{s\in[t,T]}$ under the measure $\P_{t,x}(\,\cdot\,)=\P(\,\cdot\,|X^D_{t-}=x)$.   
\end{remark}

Denote by $\LL$ the differential operator 
\begin{align}\label{def:L}
\LL=\partial_t + \frac{\sigma^2}{2}\partial_x^2+\mu \partial_x-r.
\end{align}
We have the following verification theorem.

\begin{theorem}
\label{verification}
{\bf (Verification).}
Let a function $v\in\mathcal C([0,T]\times[0,\infty))\cap\mathcal C^{1,2}([0,T)\times(0,\infty))$ be given. Assume that
\begin{itemize}
\item [(i)]
$\max\{\mathcal L v, 1-v_x\}= 0$ on $[0,T)\times(0,\infty)$;
\item[(ii)]
$v(t,0)=0$ for all $t\in[0,T]$;
\item[(iii)]
$v(T,x)=x$ for $x\in[0,\infty)$.
\end{itemize}
Further assume that there exists a continuous function $a:[0,T]\to[0,\infty)$ with $a(T)=0$ such that 
\begin{itemize}
\item[(iv)]
$\mathcal L v=0$ for $(t,x)\in[0,T)\times(0,\infty)$ with $0<x<a(t)$;
\item[(v)]
$v_x=1$ for $(t,x)\in[0,T]\times(0,\infty)$ with $x\geq a(t)$.
\end{itemize}
Then $V=v$. Moreover, 
\begin{align}\label{optcont}
D^a_s:=\sup_{0\leq u\leq s}(x + \mu u + \sigma B_u-a(t+u))^+
\end{align}
is an optimal strategy in the sense that
\[V(t,x)=  \E_x\left[\int_{0-}^{\gamma^{D^a}\wedge (T-t)}e^{-rs}dD^a_s\right].\]
\end{theorem}

\begin{proof}
Fix $(t,x)\in[0,T)\times(0,\infty)$. For a given dividend strategy $D\in\mathcal A$, denote by $D^c$ the continuous part of $D$, and
let, for $\ep>0$, 
\[\gamma^\ep:=\inf\{s\geq 0:X^D_s\leq \ep\}\wedge (T-t-\ep).\] 
Then
\begin{eqnarray*}
e^{-r\gamma^\ep}v(t+\gamma^\ep,X^D_{\gamma^\ep}) &=& v(t,x)
+
\int_0^{\gamma^\ep} e^{-rs}\mathcal L v(t+s,X^D_{s-})\, ds \\
&&
-\int_0^{\gamma^\ep}e^{-rs}v_x(t+s,X^D_{s-})\,dD^c_s\\
&&  + \sum_{0\leq s\leq \gamma^\ep} e^{-rs}\left(v(t+s,X^D_s)-v(t+s,X^D_{s-})\right)\\ 
&&
+\int_0^{\gamma^\ep} e^{-rs}\sigma v_x(t+s,X^D_{s-})\,dB_s.
\end{eqnarray*}
Note that $v_x$ is bounded on the set $[0,T-\ep]\times[\ep,\infty)$ (recall (v)), so the last integral is a (stopped) martingale.
Taking expected values of both sides gives 
\begin{eqnarray}
\label{exp}
\notag\E_x \left[e^{-r\gamma^\ep}v(t+\gamma^\ep,X^D_{\gamma^\ep}) \right]
&=& v(t,x) 
+ \E_x\left[\int_0^{\gamma^\ep} e^{-rs}\mathcal L v(t+s,X^D_{s-})\, ds \right]\\
&& \hspace{-20mm}
-\E_x\left[\int_0^{\gamma^\ep}e^{-rs}v_x(t+s,X^D_{s-})\,dD^c_s\right]\\
\notag
&& \hspace{-20mm} 
+\E_x\left[\sum_{0\leq s\leq \gamma^\ep} e^{-rs}\left(v(t+s,X^D_s)-v(t+s,X^D_{s-})\right)\right].
\end{eqnarray}
We notice that (i) implies $v_x\ge 1$ and therefore $v\geq 0$ follows from (ii). Using also $\LL v\leq 0$ and rewriting 
\[v(t+s,X^D_s)-v(t+s,X^D_{s-})=-\int_{0}^{\Delta D_s}v_x(t+s,X^D_{s-}-y)dy\]
yields
\[v(t,x)\geq \E_x\left[\int_{0-}^{\gamma^\ep}e^{-rs}\,dD_s\right].\]
Letting $\ep\to 0$ gives 
\[v(t,x)\geq \E_x\left[\int_{0-}^{\gamma^D\wedge(T-t)}e^{-rs}\,dD_s\right],\]
and since $D\in\mathcal A$ is arbitrary, $v\geq V$.

To prove the opposite inequality, let $D^a$ be the strategy given by
\[D^a_s=\sup_{0\leq u\leq s}(R_u-a(t+u))^+,\] 
where $R_u=x + \mu u + \sigma B_u$. Then $X^{D^a}_s=R_s-D^a_s$, and we notice that $(X^{D^a},D^a)$ is the solution of the Skorokhod reflection problem at the boundary $a(\,\cdot\,)$. Therefore (iv), (v) and \eqref{exp} give
\[v(t,x)=\E_x \left[e^{-r\gamma^\ep}v(t+\gamma^\ep,X^{D^a}_{\gamma^\ep})\right]+
\E_x\left[\int_{0-}^{\gamma^\ep}e^{-rs}\,dD^a_s\right].\]
Recall that $a$ is continuous with $a(T)=0$. By continuity of $v$, and since $x\mapsto v(t,x)$ is increasing due to (i), we have by Dini's theorem 
that $0\leq v(t+\gamma^\ep,X^{D^a}_{\gamma^\ep})\leq h(\ep)$ for some function $h$ with $h(\ep)\to 0$ as $\ep\to 0$.
Thus
\[v(t,x)\leq h(\ep)+
\E_x\left[\int_{0-}^{\gamma^\ep}e^{-rs}\,dD^a_s\right]\to \E_x\left[\int_{0-}^{\gamma^{D^a}\wedge(T-t)}e^{-rs}\,dD^a_s\right]\]
as $\ep \to 0$. Consequently, we also have the inequality $v\leq V$, and the strategy $D^a$ is optimal.
\end{proof}

\begin{remark}
As noted above, the inequality $V\geq x$ always holds. Moreover, if $\mu\leq 0$, then actually $V=x$, and the optimal strategy is to
immediately distribute the whole capital as dividends. 
To see this, notice that with $v(t,x)=x$ and $a(t)=0$ we have $\mathcal L v=\mu -rx\leq 0$, so that (i)-(v) of Theorem \ref{verification} are satisfied. From now on, we only consider the case $\mu>0$.
\end{remark}

\section{An optimal stopping problem with local time}
\label{sec:osp-intro}

Our approach to solving the optimisation problem given by \eqref{V} is to connect it to a suitable problem of optimal stopping. In order to find the correct candidate for the latter we begin by making some useful heuristic observations. 

If $V$ satisfies the variational problem cast in Theorem~\ref{verification}, and if in addition $V_t$ is continuous everywhere, then 
$V_t(t,0)=0$ for all $t\in[0,T)$ due to (ii). It follows that (iv) gives the boundary condition
\begin{align}\label{el1}
\tfrac{\sigma^2}{2}V_{xx}(t,0+)=-\mu V_x(t,0+),\qquad \text{for $t\in[0,T)$}.
\end{align}
Setting $u:=V_x$ we now notice that $u$ should solve, at least formally,
\begin{align}\label{freeb}
\left\{
\begin{array}{ll}
\mathcal L u=0 & \text{for $(t,x)\in[0,T)\times(0,\infty)$ with $0<x<a(t)$,}\\[+4pt]
u\ge 1 & \text{for $(t,x)\in[0,T)\times(0,\infty)$,}\\[+4pt]
u= 1 & \text{for $(t,x)\in[0,T)\times(0,\infty)$ with $x\ge a(t)$,}\\[+4pt]
u(T,x)=1 & \text{for all $x\in[0,\infty)$,}
\end{array}
\right.
\end{align}
with the additional boundary condition
\begin{align}\label{el2}
u_{x}(t,0+)=-\tfrac{2\mu}{\sigma^2} u(t,0+) \qquad \text{for $t\in[0,T)$.}
\end{align}
The boundary value problem \eqref{freeb} is reminiscent of the one associated to an optimal stopping problem with payoff of immediate stopping equal to $1$. Moreover, \eqref{el2} is similar to the so-called \emph{Feller's elastic boundary condition} at zero for a diffusion that lives on $\R_+$ (except for the minus sign on the right-hand side of \eqref{el2}). This observation is the key to finding the right connection to optimal stopping. 

Recall that (see for example \cite[Ch.~X, Exercise~1.15]{RY}), given a real-valued Markov process $(X_t)_{t\ge0}$ and 
and an additive functional $A$, one can construct the killed process 
\[\widetilde X_t=\left\{\begin{array}{ll}
X_t &\mbox{on }\{A_t< e\}\\
\Delta & \mbox{on }\{A_t\geq e\}.\end{array}\right.\]
Here $e$ is an exponentially distributed random variable with parameter 1 which is independent of $X$, 
$\Delta$ is a cemetery state, and  
this process has the associated semigroup $(\mathsf P_t)_{t\ge0}$ given by
\[\mathsf P_tf(x)=\E_x\left[e^{- A_t}f(X_t)\right].\]
Moreover, if $A_t=\lambda L^0_t$, where $(L^0_t)_{t\ge 0}$ is the local time of $X$ at 0, then 
the process $\widetilde X$ is a diffusion with elastic behaviour at zero, the associated semigroup is given by
\begin{align}\label{el-semg}
\mathsf P_tf(x)=\E_x\left[e^{- \lambda L^0_t}f(X_t)\right],
\end{align}  
and the infinitesimal generator coincides with the generator of $X$ on the space of functions $u(x)$ with $u'(0+)-u'(0-)= 2\lambda u(0)$.

It is then clear that our problem boils down to finding the appropriate process $X$.
Since \eqref{el2} expresses a one-sided condition, and recalling the expression for $\LL$ in \eqref{def:L}, we expect that in our case the process $X$ should be a \emph{reflected Brownian motion with drift $\mu$ and variance $\sigma^2$}, killed at a rate $r$. 
Hence, letting $X$ be a reflected Brownian motion with drift $\mu$ and variance $\sigma^2$, and $L^0$ its local time at zero, we are naturally led by \eqref{freeb} and \eqref{el-semg} to consider the optimal stopping problem
\begin{align}\label{U0}
U(t,x)=\sup_{0\le \tau\le T-t}\E_x\left[e^{\lambda L^0_\tau-r\tau}\right]
\end{align}
where $\lambda=2\mu/\sigma^2$. 

\begin{remark}
Note that in problem \eqref{U0} the presence of the local time corresponds to a \emph{creation} of the process (rather than killing) at 0 (cf.~\cite{P14}).
\end{remark}

We notice that the reflected Brownian motion with drift is traditionally defined in terms of analytical properties of its infinitesimal generator. However, it is known that a useful equivalence in law holds between $(X,L^0(X))$ and a more explicit process. In fact, given a standard Brownian motion $B$, setting $Y_t=-\mu t+ \sigma B_t$ and  
\[S_t=\sup_{0\leq s\leq t}Y_s,\]
it is shown in \cite{P06} that 
\begin{equation}
\label{X}
(X^x_t, L^0(X^x))\overset{\text{law}}{=}(x\vee S_t-Y_t,x\vee S_t-x).
\end{equation}
From now on we identify
\begin{align}\label{X2}
X^x_t:=x\vee S_t-Y_t\qquad{\mbox{and}}\qquad L^0_t(X^x)=x\vee S_t-x.
\end{align}
Using \eqref{X}-\eqref{X2} we can rewrite \eqref{U0} in a more tractable form as
\begin{equation}
\label{U}
U(t,x) = \sup_{0\leq\tau\leq T-t}\E_x\left[e^{\lambda L^0_\tau(X)-r\tau} \right]= \sup_{0\leq\tau\leq T-t}\E\left[e^{\lambda (x\vee S_\tau-x)-r\tau} \right],
\end{equation}
where the supremum is taken over all stopping times of $B$.

The following theorem describes the main findings of this article.

\begin{theorem}
\label{main}
{\bf (Connection)}
Let $U$ be the value function of the optimal stopping problem in \eqref{U}, and denote by $b:[0,T]\to \R$ the corresponding optimal stopping 
boundary from Proposition~\ref{propb} below. Then the value $V$ of the dividend problem satisfies
\[V(t,x)=\int_0^x U(t,y)\,dy,\]
and the dividend strategy $D^b_s=\sup_{0\leq u\leq s}(x+\mu u+\sigma B_u-b(t+u))^+$ is optimal.
\end{theorem}

The proof of Theorem~\ref{main} is given in Section \ref{sec:proof} below and it builds on the results of Sections~\ref{sec:osp}-\ref{sec:Ut} concerning $U$ and $b$. Additional properties of $V$ and $b$ may also be deduced from the study in Sections~\ref{sec:osp}-\ref{sec:Ut} and we summarise some of them as follows.

\begin{theorem}
\label{further}
{\bf (Properties of the value function and the optimal boundary)}
\begin{itemize}
\item[  (i)]
The value function $V$ belongs to $\C^{0,1}([0,T]\times[0,\infty))\cap \C^{1,2}([0,T)\times(0,\infty))$.
\item[ (ii)]
$V$ satisfies $\max\{\mathcal L V,1-V_x\}=0$ on $[0,T)\times(0,\infty)$.
\item[(iii)]
$x\to V(t,x)$ is concave, and $t\mapsto V(t,x)$ is non-increasing.
\item[ (iv)]
The boundary $b:[0,T)\to(0,\infty)$ describing the optimal dividend strategy is non-increasing, continuous and satisfies 
$b(T-)=0$.
\end{itemize}
\end{theorem}

\begin{remark}
It is worth observing that we may consider a more general controlled dynamic for the value of the fund, say
\begin{align}\label{XD2}
dX^D_t=\mu(X^D_t)dt+\sigma(X^D_t)dB_t-dD_t
\end{align}
for some functions $\mu$ and $\sigma$.
Then the verification result Theorem~\ref{verification} may be stated in a similar form, but with $\LL$ being the second order operator associated to $\mu(x)$ and $\sigma(x)$, and with the process $D^a$ realising the Skorokhod downwards reflection of $X^D$ at the boundary $a(\,\cdot\,)$. In this setting, the formal derivation of \eqref{freeb} requires to replace $\LL$ by $\widetilde{\LL}:=\partial_t+\tfrac{\sigma^2}{2}\partial_{xx}+(\mu+\sigma\,\sigma')\partial_x-(r-\mu')$ and \eqref{el2} by $u_x(t,0+)=-\tfrac{2\mu(0)}{\sigma^2(0)}u(t,0+)$. 

It follows that \eqref{U0} is again the natural candidate optimal stopping problem to be linked to \eqref{V}, but now 
$L^0$ is the local time at zero of a diffusion $X$ associated to $\widetilde \LL$ and reflected at zero, $\lambda=2\mu(0)/\sigma^2(0)$
and the constant killing rate $r$ should be replaced by a level-dependent rate $r-\mu'$.
Unfortunately, however, the general version of \eqref{U0} cannot be reduced to a simpler form (as in \eqref{U}) because there seems 
to be no analogue of \eqref{X}. In Sections~\ref{sec:Ux} and \ref{sec:Ut} below we exploit the explicitness of 
\eqref{X} to derive sufficient regularity of $U$
so that the verification result Theorem~\ref{verification} can be applied to $v(t,x)=\int_0^x U(t,y)\,dy$.
While we conjecture that $U=V_x$ holds also for more general absorbed diffusion processes, a full treatment of the general case
is technically more demanding and is therefore left for future studies.
\end{remark}

\begin{remark}
One may add a discounted running cost $e^{-rt}f(t,X^D_t)$ in the formulation for \eqref{V}. This would simply give rise to an additional running cost in \eqref{U0} of the form $e^{\lambda L_t-rt}f_x(t,X_t)$.
\end{remark}

\begin{remark}\label{rem:russian}
The optimal stopping problem \eqref{U} is closely connected to the Russian option with a finite horizon, 
see \cite{DKS}, \cite{E}, or \cite{P}. In fact, 
\[ U(t,x)=e^{-\lambda x}\tilde U(t,x),\]
where
\[\tilde U(t,x):=\sup_{0\leq\tau\leq T-t}\E\left[e^{\lambda (x\vee S_\tau)-r\tau} \right]\]
is the value of a Russian option written on a stock with current price 1, a historic maximal price $e^{\lambda x}$, 
volatility $\lambda\sigma$ and drift 0.
While some parts of our analysis in Section~\ref{sec:osp} below can be deduced from studies of the Russian 
option, we choose, for the convenience of the reader, to include a detailed study. It should be noticed, however, that 
we go substantially beyond the regularity results contained in \cite{DKS}, \cite{E}, or \cite{P} by proving that indeed $U$ is 
$\C^1$ on $[0,T)\times(0,\infty)$.
\end{remark}

\section{Analysis of the optimal stopping problem}
\label{sec:osp}

Choosing $\tau=0$ in \eqref{U} gives $U(t,x)\geq 1$. 
Denote by 
\[\mathcal C:=\{(t,x)\in[0,T]\times[0,\infty):U(t,x)>1\}\]
and
\[\mathcal D:=\{(t,x)\in[0,T]\times[0,\infty):U(t,x)=1\}\]
the continuation region and the stopping region, respectively. 
We notice that assumptions in \cite[Appendix D, Theorem D.12]{KS2} are satisfied in our setting. Hence, from standard optimal stopping theory the stopping time 
\begin{eqnarray}
\label{tau}
\tau^*:=\tau^*_{t,x}&:=& \inf\{s\geq 0:U(t+s, X^x_s)=1\}\\
\notag
&=& \inf\{s\geq 0:(t+s, X^x_s)\in\mathcal D\}
\end{eqnarray}
is optimal for the problem \eqref{U} in the sense that
\[U(t,x)=E\left[e^{\lambda (x\vee S_{\tau^*}-x)-r\tau^*} \right].\]
Moreover, for any $x\in\R_+$, the process 
\begin{align}\label{eq:supmart}
e^{-rs+\lambda L^0_s(X)}U(t+s, X_{s})\,,\quad s\ge0,
\end{align}
is a $\P_x$-supermartingale, and 
\begin{align}\label{eq:mart}
e^{-rs\wedge\tau^*+\lambda L^0_t(X)}U(t+s\wedge\tau^*, X_{s\wedge\tau^*})\,,\quad s\ge0,
\end{align}
is a $\P_x$-martingale.
In Proposition~\ref{propb} below, we provide some qualitative properties of the shape of $\mathcal C$.
First, however, we list a few properties of the value function $U$.

\begin{proposition}\label{propU} 
{\bf (Properties of $U$)}
The function $U:[0,T]\times[0,\infty)\to[0,\infty)$ is 
\begin{itemize}
\item[(i)]
equal to one at all points $(T,x)$, $x\in[0,\infty)$;
\item[(ii)]
non-increasing in $t$;
\item[(iii)]
non-increasing and convex in $x$;
\item[(iv)]
continuous on $[0,T]\times[0,\infty)$.
\end{itemize}
\end{proposition}

\begin{proof}
The first property follows directly from the definition of $U$, and (ii) is obvious since the set of stopping times is decreasing in $t$.

For (iii), note that the function $e^{\lambda (x\vee S_\tau-x)-r\tau}$ is a.s. non-increasing and convex in $x$ for any fixed $\tau$.
It follows that $\E\left[e^{\lambda (x\vee S_\tau-x)-r\tau} \right]$ is non-increasing and convex, and hence also $U$.

Finally, for (iv) we let $x_2>x_1\geq 0$ and $t\in[0,T]$. Then, by (iii),
\begin{eqnarray*}
0 &\leq& U(t,x_1)-U(t,x_2)\\
&\leq& \sup_{0\leq\tau\leq T-t}\E\left[e^{-r\tau}\left(e^{\lambda (x_1\vee S_\tau-x_1)}-e^{\lambda (x_2\vee S_\tau-x_2)}\right) \right]\\
&\leq& \lambda(x_2-x_1)
\sup_{0\leq\tau\leq T-t}\E\left[e^{\lambda  S_\tau} \right]\\
&\leq&  \lambda(x_2-x_1)\E\left[e^{\lambda  S_T} \right],
\end{eqnarray*}
where the second last inequality follows from the fact that $e^{\lambda(x\vee s-x)}$ is Lipschitz continuous in $x$ with constant 
$\lambda e^{\lambda s}$. This proves that $U$ is Lipschitz continuous in $x$, uniformly in $t\in[0,T]$. Thus, to prove continuity
of $U$ it suffices to check that $t\mapsto U(t,x)$ is continuous for any fixed $x\in[0,\infty)$.
To do that, fix $x$ and let $t_1\leq t_2$. Let $\tau^*=\tau^*_{t_1,x}$ be optimal for $(t_1,x)$, and define 
$\hat\tau:=\tau^*\wedge(T-t_2)$. Then
\begin{eqnarray*}
0 &\leq& U(t_1,x)-U(t_2,x)\\
&\leq& \E\left[ e^{\lambda (x\vee S_{\tau^*}-x)-r\tau^*}-e^{\lambda (x\vee S_{\hat\tau}-x)-r\hat\tau} \right]\\
&=& \E\left[\left(e^{\lambda (x\vee S_{\tau^*}-x)-r\tau^*}-e^{\lambda (x\vee S_{T-t_2}-x)-r(T-t_2)}\right)
1_{\left\{\tau^*\in(T-t_2,T-t_1]\right\}}\right]\\
&\leq& e^{-r(T-t_2)}\E\left[e^{\lambda(x\vee S_{T-t_1}-x)}-e^{\lambda(x\vee S_{T-t_2}-x)}\right],
\end{eqnarray*}
which tends to 0 as $t_2-t_1\to 0$ by dominated convergence. This completes the proof of (iv).
\end{proof}

\begin{proposition}
\label{propb}
There exists a boundary function $b:[0,T)\to (0,\infty)$ such that 
\begin{itemize}
\item[(i)]
$\mathcal C=\{(t,x)\in[0,T)\times[0,\infty):0\leq x<b(t)\}$;
\item[(ii)]
$b$ is non-increasing;
\item[(iii)]
$b$ is continuous on $[0,T]$ if one sets $b(T):=0$.
\end{itemize}
\end{proposition}

\begin{remark}
It follows that the optimal stopping time from \eqref{tau} satisfies
\[\tau^*_{t,x}=\inf\{s\geq 0:X^x_s\geq b(t+s)\}.\] 
The function $b$ is therefore referred to as {\bf the optimal stopping boundary}.
\end{remark}

\begin{proof}
The existence of a function $b:[0,T)\to[0,\infty]$ satisfying (i) is obvious from the fact that $x\mapsto U(t,x)$ is non-increasing.
To prove that $b>0$ on $[0,T)$, it suffices to check that $U(t,0)>1$ for $t\in[0,T)$. Choosing $\tau=\ep\leq T-t$ yields
\[
U(t,0) \geq \E\left[e^{\lambda  S_\ep-r\ep} \right]
\geq e^{-(r+\frac{2\mu^2}{\sigma^2})\ep}\E\left[e^{\sup_{0\leq s\leq \ep}\frac{2\mu}{\sigma}B_s}\right].
\]
By the reflection principle, 
\[\mathbb P(\sup_{0\leq s\leq \ep}B_s\in dz)=2\mathbb P(B_\ep\in dz)=\sqrt{\frac{2}{\pi \ep}}e^{-\frac{z^2}{2\ep}}dz\]
for $z\geq 0$, so
\begin{eqnarray*}
\E\left[e^{\sup_{0\leq s\leq \ep}\frac{2\mu}{\sigma}B_s}\right] &=& \sqrt{\frac{2}{\pi \ep}}\int_0^\infty e^{\frac{2\mu z}{\sigma}}e^{-\frac{z^2}{2\ep}}\,dz
= \sqrt{\frac{2}{\pi}}e^{\frac{2\mu^2\ep}{\sigma^2}}\int^{\infty}_{-\frac{2\mu\sqrt \ep}{\sigma}} e^{-z^2/2}dz\\
&\geq& e^{\frac{2\mu^2\ep}{\sigma^2}}+\frac{2\mu}{\sigma}\sqrt{\frac{2\ep}{\pi}}.
\end{eqnarray*}
Consequently,
\[U(t,0)\geq 1+\frac{2\mu}{\sigma}\sqrt{\frac{2\ep}{\pi}}+ \mathcal O (\ep)\]
as $\ep\downarrow 0$. This proves that $U(t,0)>1$, so $b(t)>0$ for $t\in[0,T)$.

Next, note that (ii) is immediate from (ii) of Proposition~\ref{propU}.
To prove that $b(t)<\infty$ we assume, to reach a contradiction, that $b(t)=\infty$ for some $t\in(0,T)$
(by time-homogeneity of the model, the assumption that $t>0$ is without loss of generality). Then, since $b$ is non-increasing, 
$b(s)=\infty$ for all $s\in[0,t]$, so given a starting point $(0,x)$, the optimal stopping time $\tau^*=\tau^*_{0,x}$ in \eqref{tau} satisfies $\tau^*\geq t$.
However, 
\begin{eqnarray*}
\E\left[e^{\lambda (x\vee S_{\tau^*}-x)-r\tau^*} \right] &\leq& e^{-rt}\E\left[e^{\lambda (x\vee S_{T}-x)} \right]
\to e^{-rt}<1
\end{eqnarray*}
as $x\to\infty$ by dominated convergence, since $e^{\lambda (x\vee S_{T}-x)}\le e^{\lambda S_T}\in L^1(\P)$.
The latter inequality contradicts the optimality of $\tau^*$. Thus $b(t)<\infty$ for all times $t\in[0,T)$.

From the continuity of $U$, the function $b$ is lower semi-continuous, and thus (ii) implies that it is right-continuous.
To finish the proof of (iii) it thus suffices to prove left-continuity on $(0,T]$. For this, assume (to reach a contradiction) that $b(t-)>b(t)$ for some
$t\in(0,T]$. For $\ep\in(0,t)$, consider the starting point $(t_\ep,x)$, where $t_\ep=t-\ep$ and $x=(b(t-)+b(t))/2$, and 
define 
\[\gamma_\ep:=\inf\{s\geq 0:X^x_s\notin (b(t),b(t-))\}\wedge\ep.\]
Notice that $\P_{t_\ep,x}(L^0_{\gamma_\ep}(X)=0)=1$, so that $x\vee S_{\gamma_\ep}-x=0$ almost surely as well.
Moreover $\P_{t_\ep,x}(\gamma_\ep\leq\tau^*)=1$, and therefore the martingale property \eqref{eq:mart} yields
\begin{eqnarray*}
U(t_\ep,x) &=& \E\left[e^{\lambda(x\vee S_{\gamma_\ep}-x)-r\gamma_\ep}U(t_\ep+\gamma_\ep, X^x_{\gamma_\ep})\right]\\
&\leq& \E\left[e^{-r\ep}1_{\{\gamma_\ep=\ep\}}\right]+ \E\left[U(t_\ep+\gamma_\ep, X^x_{\gamma_\ep})1_{\{\gamma_\ep<\ep\}}\right]\\
&\leq& e^{-r\ep}+ U(0,0)\P\left(\gamma_\ep<\ep\right)\\
&=& 1-r\ep +o(\ep)
\end{eqnarray*}
as $\ep\to 0$, which contradicts $U\geq 1$. Thus $b(t-)=b(t)$, which completes the proof of (iii). 
\end{proof}

We end this section by stating that the value function $U$ is a classical solution to 
a parabolic equation below the boundary. The proof of this fact is standard (e.g., see \cite[Theorem~2.7.7]{KS2})
and is therefore left out.

\begin{proposition}
\label{fbp}
The value function $U$ belongs to $\C^{1,2}$ separately in the interior of the continuation set $\C$ and in the interior of the stopping set $\D$. Moreover it satisfies 
\begin{align}
\label{freeb1} &\mathcal L U(t,x)=0 & \text{for $0<x<b(t)$ and $t\in[0,T)$,}\\
\label{freeb2} &\mathcal L U(t,x)=-r & \text{for $x>b(t)$ and $t\in[0,T)$.}
\end{align}
\end{proposition}

\section{Further regularity of $U$: the spatial derivative}\label{sec:Ux}

In this section we prove that the value function $U$ is continuously differentiable in the spatial variable,
see Theorem~\ref{spatial-derivative}.

\begin{lemma}
\label{eq-tau}
Let  
\[\tau'_{t,x}:=\inf\{s\geq 0:X^x_s>b(t+s)\}\wedge (T-t).\]
Then $\tau'_{t,x}=\tau^*_{t,x}$ a.s.
\end{lemma}

\begin{remark}
The proof of this follows the proof of \cite[Lemma 6.2]{EJ}.
\end{remark}

\begin{proof}
The claim is trivial for $(t,x)$ such that $x>b(t)$ so we fix $(t,x)\in[0,T)\times(0,\infty)$ with $x\le b(t)$. Since 
\[\tau^*:=\tau^*_{t,x}=\inf\{s\geq 0:X^x_s\geq b(t+s)\}\wedge (T-t),\]
we have $\tau':=\tau'_{t,x}\geq \tau^*$. Moreover we notice that 
\begin{equation}
\label{tau'}
\tau'_{t,b(t)}=0\qquad\text{$\P$-a.s.}
\end{equation}
due to the monotonicity of $b$ and well-known properties of Brownian motion.

To prove also that $\tau'_{t,x}\leq \tau^*$ we introduce $Z_{s,s'}=\sup_{s\leq u\leq s'}(X^x_u-b(t+u))$ for any $0<s<s'\leq T-t$.
We claim that for arbitrary but fixed $0<s_1<s_2\leq T-t$ one has
\begin{equation}
\label{s1s2}
\P(Z_{s_1,s_2}=0)=0.
\end{equation}
Let $\tau_{1}:=\inf\{u\geq s_1:X^x_u=b(t+u)\}\wedge (T-t)$.
Then
\begin{align}\label{Z1}
\P(Z_{s_1,s_2}=0) = \P(Z_{s_1,s_2}=0, \tau_{1}\in[s_1,s_2))+\P(Z_{s_1,s_2}=0, \tau_{1}=s_2)
\end{align}
because on the event $\{Z_{s_1,s_2}=0\}$ it must be $\tau_1\in[s_1,s_2]$.

For the first term of the expression on the right-hand side of \eqref{Z1} we have by continuity of $X$ and $b$ that
\begin{align*}
\P(&Z_{s_1,s_2}=0, \tau_{1}\in[s_1,s_2))\\
=&\P(Z_{s_1,s_2}=0, \tau_{1}\in[s_1,s_2), X^x_{\tau_1}=b(t+\tau_1))\\
=&\E\left[1_{\{\tau_{1}\in[s_1,s_2),\, X^x_{\tau_1}=b(t+\tau_1)\}}\,\P(Z_{s_1,s_2}=0|\cF_{\tau_1})\right]\\
=&\E\left[1_{\{\tau_{1}\in[s_1,s_2),\, X^x_{\tau_1}=b(t+\tau_1)\}}\,\P(Z_{s_1,\tau_1}\vee Z_{\tau_1,s_2}=0|\cF_{\tau_1})\right]\\
=&\E\left[1_{\{\tau_{1}\in[s_1,s_2),\, X^x_{\tau_1}=b(t+\tau_1)\}}\,\P(Z_{\tau_1,s_2}=0|X^x_{\tau_1}=b(t+\tau_1))\right]=0
\end{align*}
where we have used the strong Markov property, the fact that $Z_{s_1,\tau_1}1_{\{\tau_1<s_2\}}=0$ 
and that $\P(Z_{\tau_1,s_2}=0|X^x_{\tau_1}=b(t+\tau_1))=0$ due to \eqref{tau'}.

For the second term on the right-hand side of \eqref{Z1} we simply have
\begin{align*}
\P(Z_{s_1,s_2}=0, \tau_{1}=s_2)\le\P(X^x_{s_2}=b(t+s_2))=0
\end{align*}
since $X^x_{s_2}$ has a continuous distribution,
which finishes the proof of \eqref{s1s2}.

Now, if $\tau^*\in[s_1,s_2]$, then $Z_{s_1,s_2}\geq X^x_{\tau^*}-b(t+\tau^*)\geq 0$, so 
\begin{eqnarray*}
\P(\tau^*\in[s_1,s_2]) &=& \P(\tau^*\in[s_1,s_2], Z_{s_1,s_2}\geq 0)\\
&=& \P(\tau^*\in[s_1,s_2], Z_{s_1,s_2}> 0)\\
&=&  \P(\tau^*\in[s_1,s_2], \tau'\in[s_1,s_2]).
\end{eqnarray*}
Consequently, 
\[\P_{t,x}(s_1\leq \tau^*\leq s_2\,,\,\tau'>s_2)=0,\]
and since this holds for any rational $s_1$ and $s_2$, we have $\tau^*=\tau'$ a.s.
\end{proof}

\begin{proposition}
\label{cont-tau}
The optimal stopping time $\tau_{t,x}^*$ is continuous in $(t,x)\in[0,T]\times[0,\infty)$.
\end{proposition}

\begin{proof}
First notice that $x\mapsto X^x_t$ is a.s.~Lipschitz continuous in $x$ from \eqref{X2}, uniformly in $t\in[0,T]$.
Fix $(t,x)\in[0,T]\times[0,\infty)$ and $\omega$, and take a sequence $[0,T]\times[0,\infty)\ni(t_n,x_n)\to (t,x)$ as $n\to\infty$.

If $X^x_s>b(t+s)$ for some $s\in[0,T-t)$, i.e.~$\tau'_{t,x}\le s$, then continuity implies
\[X^{x_n}_{s}>b(t_n+s)\]
for $n$ large, so $\tau'_{t_n,x_n}\leq s$ for any such $n$. Therefore $\limsup_{n}\tau'_{t_n,x_n}\le \tau'_{t,x}$, since $s$ was arbitrary and  \begin{equation}
\label{taulimsup}
\limsup_{n\to\infty}\tau^*_{t_n,x_n}=\limsup_{n\to\infty}\tau'_{t_n,x_n}\leq \tau'_{t,x}=\tau^*_{t,x}.
\end{equation}

Next, if $s\in [0,T-t)$ is such that $X^x_u<b(t+u)$ for all $u\in[0,s]$, then $\inf_{u\in[0,s]} b(t+u)-X_u^x=:\delta>0$
by continuity of $b$ and $X^x$ in time. By Lipschitz continuity of $x\mapsto X^x_u$, 
$\inf_{u\in[0,s]} b(t+u)-X_u^{x_n}\geq\delta/2$ for 
$n$ large enough. By continuity (in time), this implies that $\inf_{u\in[0,s]} b(t_n+u)-X_u^{x_n}>0$ for large $n$.
Consequently, $\tau^*_{t_n,x_n}\geq s$, so 
\[\liminf_{n\to\infty}\tau^*_{t_n,x_n}\geq \tau^*_{t,x}\]
since $s$ was arbitrary.
Together with \eqref{taulimsup}, this yields $\lim_{n\to\infty}\tau^*_{t_n,x_n}=\tau^*_{t,x}$.
\end{proof}

\begin{theorem}
\label{spatial-derivative}
The spatial derivative $U_x(t,x)$ exists at all points $(t,x)\in[0,T]\times[0,\infty)$ and is continuous on $[0,T)\times[0,\infty)$.
Moreover, it satisfies
\begin{equation}
\label{Ux}
U_x(t,x)=-\lambda\E\left[ 1_{\{S_{\tau^*}> x\}}e^{\lambda(S_{\tau^*}-x)-r\tau^*}\right].
\end{equation}
\end{theorem}

\begin{proof}
We first show that the function 
\[g(t,x):=-\lambda\E\left[ 1_{\{S_{\tau^*}> x\}}e^{\lambda(S_{\tau^*}-x)-r\tau^*}\right]\]
is continuous on $[0,T)\times[0,\infty)$. To do that, 
assume that $[0,T)\times[0,\infty)\ni (t_n,x_n)\to(t,x)\in [0,T)\times[0,\infty)$ as $n\to\infty$. By Proposition~\ref{cont-tau},
$\tau_{t_n,x_n}^*\to \tau_{t,x}^*$ a.s. Moreover, notice that 
\begin{equation}
\label{claim}
\P(S_{\tau^*_{t,x}}=x)=0.
\end{equation} 
Indeed, define $\hat\tau_{t,x}:=\inf\{s\geq 0:S_s=x\}\wedge (T-t)$, so that $\hat\tau_{t,x}$ is also the first time that $X^x$ equals zero. 
Since $s\mapsto S_s$ is increasing and $\P(S_{\hat\tau_{t,x}+u}>x)=1$ for all $u>0$, we have that $\P(S_{\tau^*_{t,x}}=x)=\P(\tau^*_{t,x}=\hat\tau_{t,x})$. However, since $b(t+s)>0$ for $s\in[0,T-t)$ and $X^x_{\hat\tau_{t,x}}=0$,  
\[\P(\tau^*_{t,x}=\hat\tau_{t,x})=\P(\tau^*_{t,x}=T-t)\leq \P(X^x_{T-t}=0) =0,\]
and hence \eqref{claim} holds.
By \eqref{claim}, 
\[1_{\{S_{\tau_{t_n,x_n}^*}> x_n\}}\to 1_{\{S_{\tau_{t,x}^*}> x\}}\]
a.s.~as $n\to\infty$, and consequently, 
\begin{eqnarray*}
g(t_n,x_n) &=& -\lambda\E\left[ 1_{\{S_{\tau_{t_n,x_n}^*}> x_n\}}e^{\lambda(S_{\tau^*_{t_n,x_n}}-x_n)-r\tau_{t_n,x_n}^*}\right]  \\
&\to& 
-\lambda\E\left[ 1_{\{S_{\tau_{t,x}^*}> x\}}e^{\lambda(S_{\tau^*_{t,x}}-x)-r\tau_{t,x}^*}\right]=g(t,x)
\end{eqnarray*}
by dominated convergence. This shows that $g$ is continuous on $[0,T)\times[0,\infty)$.

Now note that \eqref{Ux} holds for $t=T$ since $U(x,T)=1$ and $\tau^*$ in that case equals 0.

Next, since $x\mapsto U(t,x)$ is convex for $t\in[0,T)$, its right derivative exists everywhere on $[0,\infty)$ and its left 
derivative exists everywhere on $(0,\infty)$, and the set of points where the right and the left derivative differ (for a fixed $t\in[0,T)$)
is at most countable. Fix $(t,x)\in[0,T)\times(0,\infty)$ such that the right and the left (spatial) derivatives agree
at $(t,x)$, and let $\ep>0$. Denote $\tau^*=\tau^*_{t,x}$. Then 
\begin{eqnarray*}
U(t,x)-U(t,x+\ep)
&\leq& \E\left[e^{-r\tau^*}\left(e^{\lambda(x\vee S_{\tau^*}-x)}-e^{\lambda((x+\ep)\vee S_{\tau^*}-(x+\ep))}\right)\right]\\
&=&\E\left[e^{-r\tau^*}\left(e^{\lambda(S_{\tau^*}-x)}-1\right)1_{\{x<S_{\tau^*}\leq x+\ep\}}  \right]\\
&&+ \E\left[e^{-r\tau^*}\left(e^{\lambda(S_{\tau^*}-x)}-e^{\lambda(S_{\tau^*}-(x+\ep))}\right)1_{\{S_{\tau^*}> x+\ep\}}  \right]\\
&\leq&(e^{\lambda \ep}-1)\P\left(x<S_{\tau^*}\leq x+\ep\right)\\
&&+(1-e^{-\lambda\ep}) \E\left[e^{-r\tau^*}e^{\lambda(S_{\tau^*}-x)}1_{\{S_{\tau^*}> x+\ep\}}  \right].
\end{eqnarray*}
Dividing by $\ep$ and using that 
\[\lim_{\ep\downarrow 0}\P\left(x<S_{\tau^*}\leq x+\ep\right)=0,\]
we obtain that the right (spatial) derivative at $(t,x)$ satisfies
\begin{equation}
\label{liminf}
\lim_{\ep\downarrow 0}\frac{U(t,x+\ep)-U(t,x)}{\ep}\geq-\lambda\E\left[ 1_{\{S_{\tau^*}> x\}}e^{\lambda(S_{\tau^*}-x)-r\tau^*}\right].
\end{equation}

Similarly, for $\ep\in(0,x)$, 
\begin{eqnarray*}
 U(t,x)-U(t,x-\ep)
&\leq& \E\left[e^{-r\tau^*}\left(e^{\lambda(x\vee S_{\tau^*}-x)}-e^{\lambda((x-\ep)\vee S_{\tau^*}-(x-\ep))}\right)\right]\\
&=& \E\left[e^{-r\tau^*}\left(1-e^{\lambda(S_{\tau^*}-(x-\ep))}\right)1_{\{x-\ep<S_{\tau^*}\leq x\}}  \right]\\
&&+ \E\left[e^{-r\tau^*}\left(e^{\lambda(S_{\tau^*}-x)}-e^{\lambda(S_{\tau^*}-(x-\ep))}\right)1_{\{S_{\tau^*}> x\}}  \right]\\
&\leq& \left(1-e^{\lambda\ep}\right)\E\left[e^{\lambda (S_{\tau^*}-x)-r\tau^*}1_{\{S_{\tau^*}> x\}}  \right],\\
\end{eqnarray*}
so the left (spatial) derivative satisfies
\begin{equation}
\label{limsup}
\lim_{\ep\downarrow 0}\frac{U(t,x)-U(t,x-\ep)}{\ep}\leq-\lambda\E\left[ 1_{\{S_{\tau^*}> x\}}e^{\lambda(S_{\tau^*}-x)-r\tau^*}\right].
\end{equation}

Since the derivative exists at $(t,x)$, it follows from \eqref{liminf} and \eqref{limsup} that $U_x$ satisfies \eqref{Ux} at $(t,x)$.
Moreover, since $U_x(t,\cdot)$ coincides with the continuous function $g(t,\cdot)$ outside a countable set, and since the right (left) derivative
of a convex function is right (left) continuous, it follows that $U_x=g$ on $[0,T)\times[0,\infty)$, which completes the proof.
\end{proof}

\begin{corollary}{\bf (Creation condition)}
\label{bc}
The value function $U$ satisfies the boundary condition $U_x(t,0)+\lambda U(t,0)=0$ for $t<T$.
\end{corollary}

\begin{proof}
For $t\in[0,T)$ we have $\tau^*_{0,t}>0$ a.s., so $S_{\tau^*_{0,t}}>0$ a.s. Thus
\[U_x(t,0) = -\lambda\E\left[ e^{\lambda S_{\tau_{0,t}^*}-r\tau_{0,t}^*}\right] =-\lambda U(t,0).\]
\end{proof}

\begin{corollary}
The smooth fit condition holds, i.e. $U_x(t,b(t))=0$ for $t\in[0,T)$.
\end{corollary}

\begin{proof}
This follows since $U(t,x)=1$ for $x\geq b(t)$ and $U_x$ is continuous.
\end{proof}

\section{Further regularity of $U$: the time derivative}\label{sec:Ut}

In this section we show that the time derivative of $U$ is continuous, see Theorem~\ref{time-derivative}.

\begin{lemma}
\label{Lip-in-t}
The function $U$ is Lipschitz continuous in $t$ on $[0,T_1]\times[0,\infty)$, uniformly with respect to $x$, for any $T_1\in(0,T)$.
\end{lemma}

\begin{proof}
Let $t_1,t_2\in[0,T_1]$ with $t_1< t_2$, let $x\in[0,\infty)$, and denote $\tau:=\tau^*_{t_1,x}$. Then, recalling that $U(\,\cdot\,,x)$ is decreasing, we get 
\begin{eqnarray*}
0 &\leq& U(t_1,x)-U(t_2,x)\\
&\leq& \E\left[e^{\lambda(x\vee S_\tau-x)-r\tau}-e^{\lambda(x\vee S_{\tau\wedge(T-t_2)}-x)-r(\tau\wedge(T-t_2))}\right]\\
&=& \E\left[\left(e^{\lambda(x\vee S_\tau-x)-r\tau}-  e^{\lambda(x\vee S_{T-t_2}-x)-r(T-t_2)}\right)1_{\{T-t_2<\tau\leq T-t_1\}}\right]\\
&\leq& e^{-\lambda x}e^{-r(T-t_2)}\E\left[e^{\lambda (x\vee S_{T-t_1})}-e^{\lambda (x\vee S_{T-t_2})}\right]\\
&\leq& \E\left[e^{\lambda  S_{T-t_1}}-e^{\lambda S_{T-t_2}}\right],
\end{eqnarray*}
where we for the last inequality used
\begin{align*}
e^{\lambda (x\vee S_{T-t_1})}-e^{\lambda (x\vee S_{T-t_2})}= 1_{\{x<S_{T-t_1}\}}\left(e^{\lambda S_{T-t_1}}-
e^{\lambda(x\vee S_{T-t_2})}
\right).
\end{align*}
By explicit formulas, see e.g. \cite[Section 3.5.C]{KS1},
\[\P(S_t\geq z)=\int_0^t\frac{z}{\sigma\sqrt{2\pi s^3}} e^{-\frac{(z+\mu s)^2}{2\sigma^2 s}}\,ds,\]
so 
\[\P(S_t\in dz)=\int_0^t\frac{1}{\sigma\sqrt{2\pi s^3}}\left(\frac{z+\mu s}{\sigma^2 s}z-1\right) 
e^{-\frac{(z+\mu s)^2}{2\sigma^2 s}}\,ds\,dz.\]
Thus
\begin{eqnarray*}
\E\left[e^{\lambda  S_{t}}\right] &=& \int_0^\infty e^{\lambda z}
\int_0^t\frac{1}{\sigma\sqrt{2\pi s^3}}\left(\frac{z+\mu s}{\sigma^2 s}z-1\right) e^{-\frac{(z+\mu s)^2}{2\sigma^2 s}}\,ds\,dz\\
&=:& f(t)
\end{eqnarray*}
and 
\begin{eqnarray*}
f'(t) &=& \int_0^\infty e^{\lambda z}
\frac{1}{\sigma\sqrt{2\pi t^3}}\left(\frac{z+\mu t}{\sigma^2 t}z-1\right) e^{-\frac{(z+\mu t)^2}{2\sigma^2 t}}\,dz.
\end{eqnarray*}
Since $f'(t)$ is bounded for $t\in[T-T_1,T]$, the function $U$ is Lipschitz on $[0,T_1]\times[0,\infty)$
\end{proof}

The following theorem is the main result of the current section.

\begin{theorem}
\label{time-derivative}
The time derivative $U_t$ is continuous on $[0,T)\times (0,\infty)$.
\end{theorem}

\begin{proof}
Fix $T_0\in(0,T)$, let $t_1,t_2\in[0,T_0]$ with $t_1<t_2$ and $T_0+t_2-t_1<T$, and let $x\in[0,\infty)$. Define $\tau:=\tau_{t_1,x}^*\wedge(T_0-t_2)$, 
and note that $U(t_1+\tau,X^x_\tau)=U(t_2+\tau,X^x_\tau)=1$ on the set where $\tau=\tau^*_{t_1,x}$ thanks to (ii) in Proposition~\ref{propU}.
Consequently, using \eqref{eq:supmart} and \eqref{eq:mart}, we have
\begin{align*}
0 \leq & U(t_1,x)-U(t_2,x) \\
\leq& \E\left[ e^{\lambda(x\vee S_\tau-x)-r\tau}\left(U(t_1+\tau,X^x_{\tau})-U(t_2+\tau,X^x_{\tau})\right)\right]\\
=& 
\E\left[ e^{\lambda(x\vee S_\tau-x)-r\tau}\left(U(t_1+\tau,X^x_{\tau})-U(t_2+\tau,X^x_{\tau})\right)1_{\{\tau=T_0-t_2\}}\right]\\
\leq& \E\Big[ e^{\lambda S_{T_0-t_2}-r(T_0-t_2)}\Big(U(T_0+t_1-t_2,X^x_{T_0-t_2})\\
&\hspace{+3.5cm}-U(T_0,X^x_{T_0-t_2})\Big)
1_{\{\tau=T_0-t_2\}}\Big]\\
\leq& C_0(t_2-t_1)\E\left[ e^{\lambda S_{T}}1_{\{\tau=T_0-t_2\}}\right],
\end{align*}
where $C_0$ is a Lipschitz constant of $t\mapsto U(t,x)$, see Lemma~\ref{Lip-in-t}, which depends on $T_0$.
In particular, for $(t,x)\in\mathcal C$ with $t\leq T_0$, we have 
\begin{eqnarray*}
0 &\leq& U_t(t,x)\leq C_0 \E\left[ e^{\lambda S_{T}}1_{\{\tau^*_{t,x}\geq T_0-t\}}\right]\\
&\leq& C_0\sqrt{\E\left[e^{2\lambda S_T}\right]}\sqrt{\P \left(\tau^*_{t,x}\geq T_0-t\right)}
\end{eqnarray*}
by the Cauchy-Schwarz inequality. By Proposition~\ref{cont-tau}, $\tau_{t,x}^*\to \tau_{t_0,b(t_0)}^*=0$ a.s. as
$(t,x)\to (t_0,b(t_0))$ for $t_0< T_0$. Consequently, $\P \left(\tau^*_{t,x}\geq T_0-t\right)\to 0$, and hence
$U_t(t,x)\to 0$ as $(t,x)\to (t_0,b(t_0))$. Since $U_t$ is continuous in the interior of $\mathcal D$ and in $\mathcal C$, this
shows that $U_t$ is continuous on $[0,T_0)\times(0,\infty)$.
Consequently, since $T_0$ is arbitrary, this shows that $U_t$ is continuous on $[0,T)\times (0,\infty)$.
\end{proof}

\begin{remark}
The $C^1$-differentiability of the value function is typically referred to as the `smooth fit' condition in optimal stopping theory.
This is a well-known condition that can be utilized in perpetual problems to produce a candidate solution, which then can be verified 
to equal the value function. While the smooth fit condition is generally believed to hold also for time-dependent problems (with sufficiently 
smooth underlying data), a formal verification of this fact is often lacking. In fact, in most studies of time-dependent
optimal stopping problems it is only shown that smooth fit holds in 
the {\em spatial} variable for each fixed time, 
see e.g. \cite{DKS}, \cite{E} and \cite[Lemma 2.7.8]{KS2}.
 
In that respect, Theorems~\ref{spatial-derivative} and \ref{time-derivative} 
go beyond established theory for optimal stopping problems.
\end{remark}

\section{Proof of Theorem~\ref{main}}\label{sec:proof}
Define a new function $v:[0,T]\times[0,\infty)\to[0,\infty)$ by 
\[v(t,x)=\int_0^x U(t,y)\,dy.\]
From Theorem~\ref{spatial-derivative} and Theorem~\ref{time-derivative} it is immediate to see that $v\in\C^{1,2}([0,T)\times(0,\infty))$ and it is continuous everywhere. Moreover $t\mapsto v(t,x)$ is decreasing and $x\mapsto v(t,x)$ is concave due to Proposition \ref{propU}. 

Now we want to show that $(v,b)$ solves $(i)$ to $(v)$ of Theorem \ref{verification} so that $v=V$ and $D^b$ as in \eqref{optcont} is optimal.

Since $U\ge 1$ and $U(T,x)=1$ we have $v_x\ge 1$ and $v(T,x)=x$. For $t\in[0,T)$ and $0<x<b(t)$ and recalling \eqref{freeb1} we obtain
\begin{align*}
&(v_t+\mu v_x+\tfrac{\sigma^2}{2}v_{xx}-rv)(t,x)\\
&=\int_0^x (U_t-rU)(t,y)dy+\mu U(t,x)+\tfrac{\sigma^2}{2}U_x(t,x)\\
&=-\int_0^x (\tfrac{\sigma^2}{2}U_{xx}+\mu U_x)(t,y)dy+\mu U(t,x)+\tfrac{\sigma^2}{2}U_x(t,x)\\
&=\tfrac{\sigma^2}{2}U_x(t,0)+\mu U(t,0)=0
\end{align*} 
where the last equality follows from the creation condition at zero of Corollary~\ref{bc}. Repeating the same calculation for $t\in[0,T)$ and $x\ge b(t)$ and using \eqref{freeb2} we then find $v_t+\mu v_x+\tfrac{\sigma^2}{2}v_{xx}-rv\le 0$. Therefore we have verified $(i)$ to $(v)$ of Theorem~\ref{verification}, since $(ii)$ is obviously true.

\section{Concluding remarks}
\label{sec:remarks}

The connection established in Theorem~\ref{main} enables the use of techniques from optimal stopping
theory in the study of the dividend problem. For example, the precise asymptotic behaviour of the boundary can
be derived, and the boundary can be characterised in terms of an integral equation. Both these results have
been derived in the context of Russian options (recall the remark at the end of Section~\ref{sec:osp-intro})
and we only provide their statements.

\subsection{Asymptotic behaviour of the boundary (\cite{E})}

The optimal stopping boundary satisfies 
\[\lim_{t\uparrow T} \frac{b(t)}{\sqrt{ (T-t)\ln\frac{1}{(T-t)}}}=\sigma,\]
where we notice that our $b(t)$ is equal to $\lambda^{-1}\ln a_t$ with $a_t$ as in \cite{E}.
This asymptotic behaviour is the first term in an expansion that was found in \cite{G2}. It is important to remark that in \cite{G2} the author needs an a priori assumption regarding additional regularity of $b$. Here, on the other hand, relying on the asymptotic formula from 
\cite{E}, together with the established connection between the Russian option and the dividend problem, we do not require any additional assumptions.

\subsection{An integral equation for the boundary (\cite{P})}

The boundary $b$ solves the integral equation
\begin{equation}
\label{inteqn}
1=\E\left[e^{\lambda(b(t)\vee S_{T-t}-b(t))-r(T-t)}\right]+r\int_0^{T-t}e^{-rs}\P\left(X^{b(t)}_s\geq b(t+s)\right)\,ds.
\end{equation}
The above formula may be deduced from \cite{P} via algebraic transformations. Alternatively, it can be directly verified by applying Dynkin's formula to $e^{\lambda L^0_t(X)-rt}U(t,X_t)$, since $U\in \C^1$ with $U_{xx}$ bounded, and using \eqref{freeb1}, \eqref{freeb2} and $U(t,b(t))=1$. 
Moreover, $b$ is the unique solution of \eqref{inteqn} in the class of continuous and positive functions.

\subsection{The infinite horizon case}

Finally, we remark that the connection suggested by Theorem~\ref{main}
between the infinite horizon dividend problem (see \cite{RS96}) and the perpetual optimal 
stopping problem 
\[U(x)=\sup_{\tau\geq 0}\E\left[e^{\lambda (x\vee S_\tau-x)-r\tau} \right]\]
(which is closely related to the value function of a Russian option, see \cite{SS})
also holds. Indeed, one way to see this is to do the obvious changes in the scheme of the current paper. 
Alternatively, since these problems can be solved explicitly, it is straightforward to check that
$V=U'$ by explicit calculations. Notably, however,
this connection seems unnoticed even in the perpetual setting.

\end{document}